\documentclass[]{amsart}

\usepackage{amscd,amsthm,amssymb,amsfonts,amsmath,euscript}

\theoremstyle{plain}
\newtheorem{thm}{Theorem}[section]
\newtheorem{lemma}[thm]{Lemma}

\theoremstyle{definition}

\theoremstyle{remark}
\newtheorem{remark}[thm]{Remark}

\newcommand{\nc}{\newcommand}

\def\makeop#1{\expandafter\def\csname#1\endcsname
  {\mathop{\rm #1}\nolimits}\ignorespaces}
\makeop{Hom}   \makeop{End}   \makeop{Aut}   \makeop{Isom}  \makeop{Pic} 
\makeop{Gal}   \makeop{ord}   \makeop{Char}  \makeop{Div}   \makeop{Lie} 
\makeop{PGL}   \makeop{Corr}  \makeop{PSL}   \makeop{sgn}   \makeop{Spf}
\makeop{Spec}  \makeop{Tr}    \makeop{Nr}    \makeop{Fr}    \makeop{disc}
\makeop{Proj}  \makeop{supp}  \makeop{ker}   \makeop{im}    \makeop{dom}
\makeop{coker} \makeop{Stab}  \makeop{SO}    \makeop{SL}    \makeop{SL}
\makeop{Cl}    \makeop{cond}  \makeop{Br}    \makeop{inv}   \makeop{rank}
\makeop{id}    \makeop{Fil}   \makeop{Frac}  \makeop{GL}    \makeop{SU}
\makeop{Nrd}   \makeop{Sp}    \makeop{Tr}    \makeop{Trd}   \makeop{diag}
\makeop{Res}   \makeop{ind}   \makeop{depth} \makeop{Tr}    \makeop{st}
\makeop{Ad}    \makeop{Int}   \makeop{tr}    \makeop{Sym}   \makeop{can}
\makeop{length}\makeop{SO}    \makeop{torsion} \makeop{GSp} \makeop{Ker}
\makeop{Adm}   \makeop{Mat}
\def\makebb#1{\expandafter\def
  \csname bb#1\endcsname{{\mathbb{#1}}}\ignorespaces}
\def\makebf#1{\expandafter\def\csname bf#1\endcsname{{\bf
      #1}}\ignorespaces} 
\def\makegr#1{\expandafter\def
  \csname gr#1\endcsname{{\mathfrak{#1}}}\ignorespaces}
\def\makescr#1{\expandafter\def
  \csname scr#1\endcsname{{\EuScript{#1}}}\ignorespaces}
\def\makecal#1{\expandafter\def\csname cal#1\endcsname{{\mathcal
      #1}}\ignorespaces} 

\def\doLetters#1{#1A #1B #1C #1D #1E #1F #1G #1H #1I #1J #1K #1L #1M
                 #1N #1O #1P #1Q #1R #1S #1T #1U #1V #1W #1X #1Y #1Z}
\def\doletters#1{#1a #1b #1c #1d #1e #1f #1g #1h #1i #1j #1k #1l #1m
                 #1n #1o #1p #1q #1r #1s #1t #1u #1v #1w #1x #1y #1z}
\doLetters\makebb   \doLetters\makecal  \doLetters\makebf
\doLetters\makescr 
\doletters\makebf   \doLetters\makegr   \doletters\makegr

     \def\qed{\qedmark\medbreak}%
\def\qedmark{{\enspace\vrule height 6pt width 5pt depth 1.5pt}}%

\normalsize

\makeop{Bl}

\newcommand{\A}{\mathbb A}    

\newcommand{\<}{\langle}   
\renewcommand{\>}{\rangle} 

\nc{\embed}{\hookrightarrow}

\nc{\ol}{\overline}
\nc{\wt}{\widetilde}
\nc{\opp}{\mathrm{opp}}

\def\wh{\widehat}

\makeop{Ram}
\makeop{Rep}

\begin{document}
\renewcommand{\thefootnote}{\fnsymbol{footnote}}
\setcounter{footnote}{-1}
\numberwithin{equation}{section}


\title{A remark on Chevalley's ambiguous class number formulas}
\author{Chia-Fu Yu}
\address{
Institute of Mathematics, Academia Sinica and NCTS (Taipei Office)\\
Astronomy Mathematics Building \\
No. 1, Roosevelt Rd. Sec. 4 \\ 
Taipei, Taiwan, 10617} 
\email{chiafu@math.sinica.edu.tw}



\date{\today}
\subjclass[2010]{} 


\begin{abstract}
  In this note we remark that
  Chevalley's ambiguous class number formula 
  is an immediate consequence of the Hasse norm theorem, the local and
  global norm index theorems for cyclic extensions.     
\end{abstract} 

\maketitle


\section{Introduction}
\label{sec:01}

The ambiguous ideal classes are the Galois invariants of 
the class group of a cyclic extension of a number field. 
In \cite{chevalley} Chevalley gave a formula for the
ambiguous class number.
Besides Chevalley's original paper, one can also find 
for Chevalley's ambiguous class formula 
in Gras' book (see \cite[p.~178, p.~180]{gras:cft}).
Lemmermeyer \cite{lemmermeyer:amb} supplied a modern and
elementary proof as its proof does not seem to appear in most textbooks
of algebraic number theory. 
In this note we observe that this formula can be also 
deduced very easily from 
the Hasse norm theorem, the local and
global norm index theorems for cyclic extensions.  
The proofs of these standard theorems can be found in most 
textbooks of algebraic number theory, for example in Lang's book 
\cite{lang:ant}. We should note that our proof is not completely
independent of the proofs appeared before as 
the $Q$-machine of Herbrand is also applied in the proof of the local
norm index theorem.  

We now state the ambiguous class number formula. 
Consider a cyclic extension $K/k$ of number fields 
with cyclic Galois group $G=\Gal(K/k)=\<\sigma\>$ of generator $\sigma$.
Denote by $\gro$ and $\grO$ the ring of integers of
$k$ and $K$, respectively. 
Let $\infty$ and $\infty_r$ 
(resp. $\wt \infty$ and $\wt \infty_r$) denote the set
of infinite and real places of $k$
(resp. of $K$), respectively, and $\A_k$ (resp. $\A_K$) the adele ring
of $k$ (resp. $K$). Let $r_k:\wt
\infty \to \infty$ denote the restriction to $k$.
A real cycle is a cycle which is supported in the set of real places 
\cite[p.~123]{lang:ant}. We may identify a real cycle 
with its support, which is a subset of real places.    

Suppose that  $\wt \grc$ is a real cycle on $K$ which is stable under the
$G$-action. Denote by
\begin{equation}
  \label{eq:ClKc}
  \Cl(K,\wt \grc):=\frac{\A_K^\times}{K^\times \wh \grO^\times
  K_\infty(\wt \grc)^\times}
\end{equation}
the narrow ideal class group of $K$ with respect to $\wt \grc$, where $\wh
\grO$ is the profinite completion of $\grO$, and 
$ K_\infty(\wt \grc)^\times =\{a=(a_w) \in K_\infty^\times \mid a_w
>0 \quad \forall\, w| \wt \grc \}.$ 
Similarly one defines $\Cl(k,\grc)$ for any real cycle $\grc$ on $k$. 
The group $G$ acts on the finite abelian group $\Cl(K,\wt \grc)$. 
Its $G$-invariant subgroup $\Cl(K,\wt \grc)^G$ is 
called the {\it ambiguous ideal class group} (with respect to $\wt \grc$). 

Let $\grc$ be the real cycle on $k$ such that $\infty_r-\grc=r_k(\wt
\infty_r-\wt \grc)$,  and let $\grc_0:=r_k(\wt \grc)$. 
One has $\grc=\grc_0
\infty_r^c$, where $\infty_r^c$ is the set of real places of $k$ which does
not split completely in $K$. Let
$N_{K/k}$ denote the norm map from $K$ to $k$. The cycle $\grc$ is
determined by the property 
$N_{K/k}(K_\infty(\wt \grc)^\times)=k_\infty(\grc)^\times$. 
Put
$\gro(\grc)^\times:=\gro^\times \cap
i_\infty^{-1}(k_\infty(\grc)^\times)$, where $i_\infty:k^\times \to
k_\infty^\times$ is the diagonal embedding.    
Denote by
$V_f$ the set of finite places of $k$. Let $e(v)$  denote
the ramification index of any place $w$ over $v\in V_f$. 

\begin{thm}\label{form_ClG}
  One has 
  \begin{equation}
    \label{eq:form_ClG}
    \# \Cl(K,\wt \grc)^G=\frac{\#\Cl(k,\grc) \prod_{v\in V_f}
    e(v)}{[K:k][\gro(\grc)^\times : \gro(\grc)^\times \cap
    N_{K/k}(K^\times)]}. 
  \end{equation}
\end{thm}

When $\wt \grc=\wt \infty_r$, we get the restricted version of the
formula stated in \cite[p.~178]{gras:cft}. 
When $\wt \grc=\emptyset$, using an elementary fact
\[ \# \Cl(k,\infty_r^c)=\frac{h(k)\cdot
  2^{|\infty_r^c|}}{[\gro^\times:\gro(\infty_r^c)^\times]}, \]
we get the ordinary version of the formula stated  
in \cite[p.~180]{gras:cft}. 

\section{Proof of Theorem~\ref{form_ClG}}
\label{sec:02}

Define the norm ideal class group $N(K,\wt \grc)$ by
\begin{equation}
  \label{eq:NKc}
   N(K,\wt \grc):=\frac{N_{K/k}(\A_K^\times)}
   {N_{K/k}(K^\times \wh \grO^\times K_\infty(\wt \grc)^\times)}.
\end{equation}
Consider the commutative diagram of two short exact sequences (by
Hilbert's Theorem 90)
\begin{equation}
  \label{eq:N_CD}
  \begin{CD}
  1 @>>> \A_K^{\times 1-\sigma}\cap U @>>> U @>{N_{K/k}}>>
  N_{K/k}(U) @>>> 1 \\
  @. @VVV @VVV @VVV \\
  1 @>>> \A_K^{\times 1-\sigma} @>>> \A_K^\times @>{N_{K/k}}>>
  N_{K/k}(\A_K^\times) @>>> 1,  \\    
  \end{CD}
\end{equation}
where $U=K^\times \wh \grO^\times K_\infty(\wt \grc)^\times$. The
snake lemma gives the short exact sequence
\begin{equation}
  \label{eq:Cl_N}
  \begin{CD}
    1 @>>> \Cl(K,\wt \grc)^{1-\sigma} @>>> \Cl(K,\wt \grc) @>>>
    N(K,\wt \grc) @>>> 1 
  \end{CD}
\end{equation}
as one has an isomorphism $\A_K^{\times 1-\sigma}/(\A_K^{\times 1-\sigma}\cap U)\simeq  \Cl(K,\wt \grc)^{1-\sigma}$. 
On the other hand we have the short exact sequence 
\begin{equation}
  \label{eq:ClG}
  \begin{CD}
    1 @>>> \Cl(K,\wt \grc)^G @>>> \Cl(K,\wt \grc) @>>>
    \Cl(K,\wt \grc)^{1-\sigma} @>>> 1,
  \end{CD}
\end{equation}
which with (\ref{eq:Cl_N}) shows the following result.
\begin{lemma}\label{ClG=N}
  We have $\# \Cl(K,\wt \grc)^G= \# N(K,\wt \grc)$.
\end{lemma}


Define
\[ \Cl(k,\grc, \grO):=\frac{\A_k^\times}{k^\times
   k_\infty(\grc)^\times N_{K/k}(\wh \grO^\times)}. \]

\begin{lemma}\label{N_Cl}
   The group $N(K,\wt \grc)$ is isomorphic to a subgroup $H\subset
   \Cl(k,\grc, \grO)$ of index $[K:k]$. 
\end{lemma}
\begin{proof}
  Put $A:=N_{K/k}(\A_K^\times)$,
  $B:=N_{K/k}(K^\times \wh \grO^\times K_\infty(\wt \grc)^\times)$, 
$C:=k^\times$ and $H:=CA/CB$. The group $H$ is a subgroup in
$\Cl(k,\grc,\grO)$, which is 
of index $[K:k]$ by the global norm index theorem 
\cite[p.~193]{lang:ant}.
One has $A\cap C=N_{K/k}(K^\times)\subset B$ 
by the Hasse norm theorem \cite[p.~195]{lang:ant}.
The lemma follows from 
\[ N(K,\wt \grc)=A/B = A/(A\cap C)B \simeq CA/CB=H. \text{\qed} \]

\end{proof}

Consider the exact sequence
\begin{equation}
  \label{eq:final_exact}
  \begin{CD}
    1 @>>> \frac{\gro(\grc)^\times}{\gro(\grc)^\times\cap N(\wh
    \grO^\times)}@>>> \frac{\wh \gro^\times}{N(\wh \grO^\times)} @>>>
    \Cl(k,\grc, \grO) @>>> \Cl(k,\grc) @>>> 1.  
  \end{CD}
\end{equation}
It is easy to see 
$\gro(\grc)^\times\cap N_{K/k}(\wh \grO^\times)
=\gro(\grc)^\times\cap N_{K/k}(K^\times)$ from the Hasse
    norm theorem. The local norm index theorem 
\cite[p.~188, Lemma 4]{lang:ant}
gives     
\begin{equation}
      \label{eq:loc_norm_ind}
\#\left(\frac{\wh \gro^\times}{N(\wh \grO^\times)}\right) 
   =\prod_{v\in V_f} e(v).          
\end{equation}
Combining Lemma~\ref{N_Cl}, (\ref{eq:final_exact}) and 
(\ref{eq:loc_norm_ind}) we get
\begin{equation}
  \label{eq:form_N}
  \# N(K,\wt \grc)=\frac{\#\Cl(k,\grc,\grO)}{[K:k]}=
   \frac{\#\Cl(k,\grc) \prod_{v\in V_f}
    e(v)}{[K:k][\gro(\grc)^\times : \gro(\grc)^\times \cap
    N_{K/k}(K^\times)]}.
\end{equation}
Theorem~\ref{form_ClG} follows from Lemma~\ref{ClG=N} and
(\ref{eq:form_N}). \qed
 
\begin{remark}
  We do not know whether $\Cl(K,\wt \grc)^G$ and $N(K,\wt \grc)$ are
  isomorphic as abelian groups or whether 
  there is a natural bijection between
  them. When $[K:k]=2$ and $\# \Cl(K,\wt
  \grc)^{1-\sigma}$ is odd, we show that there is a natural
  isomorphism  
  \begin{equation}
    \label{eq:Cl=N}
    N(K,\wt \grc) \simeq \Cl(K,\wt \grc)^G. 
  \end{equation}
The map $1-\sigma: \Cl(K,\wt \grc) \to \Cl(K,\wt \grc)^{1-\sigma}$
restricted to $\Cl(K,\wt \grc)^{1-\sigma}$ is the squared map 
${\rm Sq}$, which is an isomorphism from our assumption. 
The inverse of ${\rm Sq}$ defines a section of (\ref{eq:ClG}), 
and hence an isomorphism 
$\Cl(K,\wt \grc)\simeq \Cl(K,\wt \grc)^G\oplus \Cl(K,\wt
\grc)^{1-\sigma}$. The assertion (\ref{eq:Cl=N}) then follows.  
\end{remark}
\section*{Acknowledgments}

The present work is done while the author's stay in the Max-Planck-Institut
f\"ur Mathematik. He is grateful to the Institut for kind hospitality
and excellent working environment. 
The author is partially supported by the grants 
MoST 100-2628-M-001-006-MY4 and 103-2918-I-001-009.

\end{document}